\documentclass{amsart}
\title{Abstract matrix-tree theorem and Bernardi polynomial}
\author[Yu.\,Burman]{Yurii Burman}
\address{National Research University Higher School of Economics, 119048, 6
Usacheva str., Moscow, Russia, and Independent University of Moscow,
119002, 11 B.Vlassievsky per., Moscow, Russia}
\email{burman@mccme.ru}
\date{}

\makeatletter
\RequirePackage{amssymb}

\newcommand{\theoremName}{Theorem}
\newcommand{\lemmaName}{Lemma}
\newcommand{\corollaryName}{Corollary}
\newcommand{\statementName}{Proposition}

\newcommand{\remarkName}{Remark}
\newcommand{\exampleName}{Example}
\newcommand{\definitionName}{Definition}
\newcommand{\problemName}{Problem}
\newcommand{\proofName}{Proof}
\renewcommand{\proofname}{\proofName}
\newcommand{\answerName}{Answer}
\newcommand{\hintName}{Hint}

\theoremstyle{plain}

\newtheorem {theorem}{\theoremName}
\newtheorem {lemma}{\lemmaName}
\newtheorem {corollary}{\corollaryName}

\newtheorem {proposition}{\statementName}

\theoremstyle{remark}

\newtheorem{example}{\exampleName}

\theoremstyle{definition}

\newtheorem{definition}{\definitionName}

\let\@newpf\proof 
\let\proof\relax

\def \namepf[#1] {\@newpf[\proofname\ #1]}
\newenvironment{proof}{\@ifnextchar[{\namepf}{\@newpf[\proofname]}}{\qed\endtrivlist}

\newcounter{qst}
\@addtoreset{qst}{problem}

\def \Integer {{\mathbb Z}}

\def \Complex {{\mathbb C}}
\def \Rational {{\mathbb Q}}

\def \lnorm#1\rnorm {\vphantom{#1}\left\|\smash{#1}\right\|}
\def \lmod#1\rmod {\vphantom{#1}\left|\smash{#1}\right|}

\newcommand \bydef {\stackrel{\mbox{\scriptsize def}}{=}}

\@ifundefined{name}{\newcommand \name[1] {\mathop{\rm #1}\nolimits}}%
{\relax}

\renewcommand \phi {\varphi}
\renewcommand \rho {\varrho}

\newcommand{\DT}[1]{#1 \dots #1}
\def \Graph {\mathcal G}
\def \Undir {\mathcal Y}

\def \AC {\name{\mathfrak A}}
\def \SSC {\name{\mathfrak S}}
\let \Det=\det
\def \det {\Det\nolimits}

\numberwithin{equation}{section}

\theoremstyle{plain}

\numberwithin{theorem}{section}

\newtheorem{lemma}[theorem]{Lemma}

\newtheorem{proposition}[theorem]{Proposition}

\newtheorem{corollary}[theorem]{Corollary}

\theoremstyle{remark}

\newtheorem{example}[theorem]{Example}

\theoremstyle{definition}

\newtheorem{definition}[theorem]{Definition}

\makeatother

\begin{document}

 \begin{abstract}
This paper is a continuation of \cite{Burman}. We prove a three-parameter 
family of identities (Theorem \ref{Th:Main}) involving a version of the 
Tutte polynomial for directed graphs introduced by Awan and Bernardi 
\cite{Bernardi}. A particular case of this family (Corollary \ref{Cr:MTT}) 
is the higher-degree generalization of the matrix-tree theorem proved in 
\cite{Burman}, which thus receives a new proof, shorter (and less direct) 
than the original one. The theory has a parallel version for undirected 
graphs (Theorem \ref{Th:MainUndir}). 
 \end{abstract}

\maketitle

\section{Definitions and main results}

The following theory has two parallel versions --- for directed and 
undirected graphs --- so let us introduce notation for both cases.

By $\Gamma_{n,k}$ we denote the set of directed graphs with $n$ vertices
numbered $1 \DT, n$ and $k$ edges numbered $1 \DT, k$; in other words, an
element $G \in \Gamma_{n,k}$ is a $k$-element sequence $([a_1,b_1] \DT,
[a_k,b_k])$ where $a_1 \DT, a_k, b_1 \DT, b_k \in \{1 \DT, n\}$. Loops
(edges $[a,a]$) and parallel edges (pairs $[a_i,b_i] = [a_j,b_j]$) are
allowed. Similarly, by $\Upsilon_{n,k}$ we denote the set of unoriented
graphs with $n$ numbered vertices and $k$ numbered edges: $\Upsilon_{n,k} =
\{(\{a_1,b_1\} \DT, \{a_k,b_k\})\}$. The forgetful map $\lmod \cdot\rmod:
\Gamma_{n,k} \to \Upsilon_{n,k}$ relates to every graph $G$ its undirected
version $\lmod G\rmod$ obtained by dropping the edge orientation: $[a,b]
\mapsto \{a,b\}$.

Denote by $\Graph_{n,k}$ (resp., $\Undir_{n,k}$) a vector space over
$\Complex$ spanned by $\Gamma_{n,k}$ (resp., $\Upsilon_{n,k}$). The
forgetful map is naturally extended to the linear map $\lmod \cdot\rmod:
\Graph_{n,k} \to \Undir_{n,k}$.

\subsection{Main theorems}\label{SSec:MainThms}

A {\em Bernardi polynomial} \cite{Bernardi} is a map $B: \Gamma_{n,k} \to
\Rational[q,y,z]$ defined as
 \begin{equation}\label{Eq:DefB}
B_G(q,y,z) = \sum_{f: \{1 \DT, n\} \to \{1 \DT, q\}}
y^{\#f_G^{>}} z^{\# f_G^{<}}
 \end{equation}
where $f_G^{>}$ (resp., $f_G^{<}$) is the set of edges $[ab]$ of $G$ such
that $f(b) > f(a)$ (resp., $f(b) < f(a)$). See \cite{Bernardi} for a
detailed analysis of the properties of $B_G$ (in particular, for the proof
of its polynomiality).

Bernardi polynomial is a directed version of the {\em full 
chromatic polynomial}, which is a map $\name{C}: \Upsilon_{n,k} \to 
\Rational[q,y]$ defined as
 \begin{equation*}
\name{C}_G(q,y) = \sum_{f: \{1 \DT, n\} \to \{1 \DT, q\}} y^{\#f_G^{\ne}}
 \end{equation*}
where $f_G^{\ne}$ is the set of edges $[ab]$ of $G$ such that $f(b) \ne 
f(a)$. The {\em classical Potts polynomial} (as defined e.g.\ in 
\cite{WelshMerino}) is related to the full chromatic polynomial by 
$Z_G(q,v) = (v+1)^k \name{C}_G(q,1/(v+1))$; see \cite{Bernardi} for 
details.

For any $G \in \Gamma_{n,k}$ (resp., $G \in \Upsilon_{n,k}$) we denote by
$\widehat{G}$ the graph $G$ with all the loops deleted (and the numbering
of the non-loop edges shifted accordingly); we have $\widehat{G} \in
\Gamma_{n,k-\ell}$ (resp., $G \in \Upsilon_{n,k-\ell}$) where $\ell$ is the
number of loops in $G$. It follows directly from the definition that
$B_G(q,y,z) = B_{\widehat{G}}(q,y,z)$; in the undirected case $Z_G(q,v) =
(v+1)^\ell Z_{\widehat{G}}(q,v)$.

The {\em universal Bernardi polynomial} is an element of
$\Graph_{n,k}[q,y,z]$ defined as
 \begin{equation*}
{\mathcal B}_{n,k}(q,y,z) \bydef \sum_{G \in \Gamma_{n,k}} B_G(q,y,z) G.
 \end{equation*}
For a polynomial $P \in \Complex[q,y,z]$ denote by $[P]_k$ the sum of terms
containing monomials $q^s y^i z^j$ with $i+j = k$ (and any $s$). The {\em
universal truncated Bernardi polynomial} is an element of
$\Graph_{n,k}[q,y,z]$ defined as
 \begin{equation*}
\widehat{{\mathcal B}}_{n,k}(q,y,z) \bydef \sum_{G \in \Gamma_{n,k}}
[B_G]_k(q,y,z) G.
 \end{equation*}
Note that $[B_G]_k = 0$ if (and only if) $G$ contains at least one loop; 
that is, $\widehat{{\mathcal B}}_{n,k}$ contains only loopless graphs. 
$\widehat{{\mathcal B}}_{n,k}$ is homogeneous of degree $k$ with respect to 
$y$ and $z$ and is not homogeneous with respect to $q$.

The {\em universal Potts polynomial} and the {\em universal truncated Potts
polynomial} are elements of $\Undir_{n,k}[q,v]$ defined, respectively, as
 \begin{align*}
{\mathcal Z}_{n,k}(q,v) &\bydef \sum_{G \in \Gamma_{n,k}}
Z_{\widehat G}(q,v) G,\\
\widehat{{\mathcal Z}}_{n,k}(q,v) &\bydef \sum_{G \in \Gamma_{n,k}
\text{ has no loops}} Z_G(q,v) G.
 \end{align*}

For any $i = 1 \DT, k$ and $p, q = 1 \DT, n$ denote by $R_{p,q;i}: 
\Gamma_{n,k} \to \Gamma_{n,k}$ the map replacing the edge number $i$ of 
every graph $G \in \Gamma_{n,k}$ by the edge $[p,q]$ carrying the same 
number $i$. Also denote by $B_i: \Graph_{n,k} \to \Graph_{n,k}$ the linear 
operator acting on the graph $G \in \Gamma_{n,k}$ such that $[a,b]$ is its 
edge number $i$ as
 \begin{equation*}
B_i(G) =  \begin{cases}
G, &a \ne b,\\
-\sum_{m \ne a} R_{a,m;i} G, &a = b.
 \end{cases}
 \end{equation*}
Following \cite{Burman} call the product $\Delta \bydef B_1 \dots B_k: 
\Graph_{n,k} \to \Graph_{n,k}$ the {\em Laplace operator}. The undirected 
version of the Laplace operator is defined as follows: if $G \in 
\Upsilon_{n,k}$ then $\Delta(G) \bydef \lmod \Delta(\Phi)\rmod$ where $\Phi 
\in \Gamma_{n,k}$ is any directed graph such that $G = \lmod \Phi\rmod$.

The main results of this paper are the following two theorems:

 \begin{theorem} \label{Th:Main}
 \begin{equation*}
\Delta {\mathcal B}_{n,k}(q,y,z) = \widehat{{\mathcal B}}_{n,k}(q,y-1,z-1).
 \end{equation*}
 \end{theorem}

\noindent and its undirected version

 \begin{theorem} \label{Th:MainUndir}
 \begin{equation*}
\Delta {\mathcal Z}_{n,k}(q,v) = (-1)^k\widehat{{\mathcal Z}}_{n,k}(q,-v).
 \end{equation*}
 \end{theorem}

\subsection{Corollaries}

\subsubsection{Universal chrmomatic polynomials}

Following \cite{Bernardi}, denote by $\chi_G^{{\ge}}$ (a {\em chromatic 
polynomial} of the directed graph $G \in \Gamma_{n,k}$) a polynomial such 
that for any $q = 1, 2, \dots$ the value $\chi_G^{{\ge}}(q)$ is equal to 
the number of mappings $f: \{1 \DT, n\} \to \{1 \DT, q\}$ such that $f(a) 
\ge f(b)$ for every edge $[ab] \in G$. Also denote by $\chi_G^{{>}}$ (a 
{\em strict chromatic polynomial} of $G$) a polynomial such that for any $q 
= 1, 2, \dots$ the value $\chi_G^{{>}}(q)$ is equal to the number of 
mappings $f: \{1 \DT, n\} \to \{1 \DT, q\}$ such that $f(a) > f(b)$ for 
every edge $[ab] \in G$. 

Comparing these definitions with the definition of the Bernardi polynomial 
in Section \ref{SSec:MainThms} one obtains the equalities:
 \begin{align*}
\chi_G^{{\ge}}(q) &= B_G(q,0,1),\\
\chi_G^{{>}}(q) &= [B_G]_k(q,0,1).
 \end{align*}
Thus, one may call the elements of $\Graph_{n,k}$
 \begin{equation*}
{\mathcal X}_{n,k}^{{\ge}}(q) \bydef {\mathcal B}_{n,k}(q,0,1) = \sum_{G 
\in \Gamma_{n,k}} \chi_G^{{\ge}}(q) G,
 \end{equation*}
and
 \begin{equation*}
{\mathcal X}_{n,k}^{{>}}(q) \bydef \widehat{{\mathcal B}}_{n,k}(q,0,1) = 
\sum_{G \in \Gamma_{n,k}} \chi_G^{{>}}(q) G
 \end{equation*}
universal chromatic polynomials. Substitution of $y=0$ and $z=1$ in Theorem 
\ref{Th:Main} yields
 \begin{corollary} \label{Cr:Chrom}
$\Delta {\mathcal X}_{n,k}^{{\ge}}(q) = (-1)^k {\mathcal 
X}_{n,k}^{{>}}(q)$.
 \end{corollary}

\subsubsection{Higher matrix-tree theorems}\label{SSec:HDet}
A graph $G \in \Gamma_{n,k}$ is called {\em acyclic} if it contains no 
oriented cycles (in particular, no loops); $G$ is called {\em totally 
cyclic} (or {\em strongly semiconnected}, following the terminology of 
\cite{Burman}) if every edge of $G$ is a part of a directed cycle. 

It is possible to make further specialization of parameters in Corollary 
\ref{Cr:Chrom} due to the following

 \begin{proposition}\label{Pp:SSCViaB}

For any $G \in \Gamma_{n,k}$
 \begin{align}
\chi_G^{{\ge}}(-1)&= \begin{cases} 
(-1)^{\beta_0(G)}, &\text{if $G$ is totally cyclic},\\
0 &\text{otherwise},
\end{cases}\label{Eq:SSCViaB}\\
\chi_G^{{>}}(-1) &= \begin{cases}
(-1)^k, &\text{if $G$ is acyclic},\\
0 &\text{otherwise}.\label{Eq:ACViaB}
 \end{cases}
 \end{align}
where $\beta_0(G)$ is the $0$-th Betti number (i.e.\ the number of 
connected components) of the graph $G$.
 \end{proposition}

For proof see \cite[Eq.\,(45) and Definition 5.1]{Bernardi}. Note that it 
follows immediately from the definition that $\chi_G^{{>}} \equiv 0$ if 
(and only if) $G$ contains an oriented cycle (e.g.\ a loop), and that 
$\chi_G^{{\ge}}(q) = q^{\beta_0(G)}$ if $G$ is totally cyclic, so one half 
of each formula is evident (but not the other half).

Consider now (following \cite{Burman}) the sum
 \begin{equation*}
\det_{n,k} \bydef \frac{(-1)^k}{k!} \chi_G^{{\ge}}(-1) = \frac{(-1)^k}{k!} 
\sum_{G \in \SSC_{n,k}} (-1)^{\beta_0(G)} G
 \end{equation*}
where by $\SSC_{n,k} \subset \Gamma_{n,k}$ we denote the set of all totally 
cyclic graphs. Thus, Corollary \ref{Cr:Chrom} specializes to
 \begin{corollary} \label{Cr:SumAll}
$\Delta \det_{n,k} = \frac{(-1)^n}{k!} \sum_{G \in \AC_{n,k}} G$.
 \end{corollary}
\noindent where $\AC_{n,k} \subset \Gamma_{n,k}$ is the set of all acyclic 
graphs. 

Corollary \ref{Cr:SumAll} admits a refinement. A totally cyclic graph may 
have isolated vertices (the ones not incident to any edge). Let $I = \{i_1 
\DT< i_s\} \subset \{1 \DT, n\}$ be a set of vertices. We call a {\em 
diagonal $I$-minor} the element
 \begin{equation}
\det_{n,k}^I \bydef \frac{(-1)^k}{k!} \sum_{G \in \SSC_{n,k}^I} 
(-1)^{\beta_0(G)} G \in \Graph_{n,k}
 \end{equation}
where $\SSC_{n,k}^I$ is the set of all totally cyclic graphs $G \in 
\Gamma_{n,k}$ such that the vertices $i_1 \DT, i_s$, and only they, are 
isolated. Similarly, denote by $\AC_{n,k}^I \subset \Gamma_{n,k}$ the set 
of all acyclic graphs such that $i_1 \DT, i_s$, and only they, are sinks 
(vertices without edges starting at them); so Corollary \ref{Cr:SumAll} now 
looks like
 \begin{equation*}
\Delta \sum_{I \subset \{1 \DT, n\}} \det_{n,k}^I = \frac{(-1)^n}{k!} 
\sum_{I \subset \{1 \DT, n\}} \sum_{G \in \AC^I_{n,k}} G.
 \end{equation*}
It follows from the definition of the Laplace operator that if $G \in 
\SSC_{n,k}^I$ then $\Delta G = \sum_H x_H H$ where $x_H \in \Integer$ and 
all $H$ have $i_1 \DT, i_s$, and only them, as sinks. Since $\SSC_{n,k}^I$ 
with different $I$ do not intersect, and the same is true for 
$\AC_{n,k}^I$, there is

 \begin{corollary}[of Corollary \ref{Cr:SumAll}] \label{Cr:MTT}
For every $I = \{i_1 \DT, i_s\} \subset \{1 \DT, n\}$ one has $\Delta 
\det_{n,k}^I = \frac{(-1)^n}{k!} \sum_{G \in \AC^I_{n,k}} G$.
 \end{corollary}
This is the abstract-matrix tree theorem \cite[Theorem 1.7]{Burman} which, 
in turn, is a higer-degree generalization of the celebrated matrix-tree 
theorem (first discovered by G.\,Kirchhoff in 1847 \cite{Kirch} and 
extended to the directed graphs by W.\,Tutte \cite{TutteMTT}).

\subsection*{Acknowledgements}

The research was funded by the Russian Academic Excellence Project `5-100'
and by the grant No.~15-01-0031 ``Hurwitz numbers and graph isomorphism''
of the Scientific Fund of the Higher School of Economics.

\section{Proofs}

A graph $H \in \Gamma_{n,m}$ is called a subgraph of $G \in \Gamma_{n,k}$ 
(notation $H \subseteq G$) if it can be obtained from $G$ by deletion of 
several edges. (When one deletes the edge number $s$ from the graph, the 
numbers of the remaining edges are preserved if they are less than $s$ and 
are lowered by $1$ if they are greater than $s$.) 

For convenience denote by $e(G)$ the number of edges of the graph $G$ (so 
$e(G) = k$ if $G \in \Gamma_{n,k}$. Proof of Theorem \ref{Th:Main} involves 
the following well-known lemma:

 \begin{lemma}[Moebius inversion formula, \cite{Rota}] \label{Lm:Moeb}
Let $f:\bigcup_k \Gamma_{n,k} \to \Complex$ be a function on the set of 
graphs with $n$ vertices, and let the function $h$ on the same set be 
defined by the equality $h(G) = \sum_{H \subseteq G} f(H)$ for every $G \in 
\Gamma_{n,k}$. Then one has $(-1)^{e(G)} f(G) = \sum_{H \subseteq G} 
(-1)^{e(H)} h(H)$.
 \end{lemma}

 \begin{proof}[of Theorem \ref{Th:Main}]
Let $\Delta {\mathcal B}_{n,k}(q,y,z) = \sum_{G \in \Gamma_{n,k}} x_G G$; 
by the definition of the Laplace operator $x_G \ne 0$ only if $G$ contains 
no loops. For a graph $H \in \Gamma_{n,k}$ the element $\Delta H \in 
\Graph_{n,k}$ contains a term $y_{G,H} G$ with $y_{G,H} \ne 0$ if and only 
if $\Phi \bydef \widehat H$ (the graph $H$ with all the loops deleted) is a 
subgraph of $G$. For any subgraph $\Phi \subseteq G$ of a loopless graph 
$G$ there exists exactly one $H \bydef L(\Phi)$ such that $\Phi = 
\widehat{H}$: every edge $[ab]$ present in $G$ but missing in $\Phi$ is 
replaced by the loop $[aa]$ in $H$. 

Eventually, the coefficient $y_{G,H}$ in this case is 
 \begin{equation*}
y_{G,H} = (-1)^{\#\text{of loops in $H$}} B_H(q,y,z) = (-1)^{k - 
e(\Phi)}B_{\Phi}(q,y,z).
 \end{equation*}
where $\Phi = \widehat{H}$. 

By \cite[Eq.~(21)]{Bernardi} one has $\sum_{\Phi \subseteq G} 
[B_\Phi]_{e(\Phi)}(q,y-1,z-1) = B_G(q,y,z)$. Applying the Moebius inversion 
formula (Proposition \ref{Lm:Moeb}) to this identity one obtains
 \begin{equation*}
x_G = \sum_{\Phi \subseteq G} y_{G,L(\Phi)}= (-1)^k \sum_{\Phi \subseteq 
G} (-1)^{e(\Phi)} B_\Phi(q,y,z) = [B_G]_k(q,y-1,z-1).
 \end{equation*}
 \end{proof}

 \begin{proof}[of Theorem \ref{Th:MainUndir}]
is similar to that of Theorem \ref{Th:Main}: again, if $\Delta {\mathcal 
Z}_{n,k}(q,v) = \sum_G x_G G$ then $G$ entering the sum have no loops. A 
contribution $y_{G,H}$ of a graph $H$ that into $x_G$ is nonzero if and 
only if $\Phi = \widehat{H}$ is a subgraph of $G$. For a subgraph $\Phi 
\subseteq G$ having $e(\Phi)$ edges there are $2^{k-e(\Phi)}$ graphs $H$ 
such that $\Phi = \widehat{H}$: every edge $[ab]$ present in $G$ but 
missing in $\Phi$ may correspond either to a loop $[aa]$ or to a loop 
$[bb]$ in $H$; recall that $a \ne b$ because $G$ is loopless. 

The contribution $y_{G,H}$ of all such graphs $H$ into $x_G$ is the same 
and is equal to $(-1)^{k-e(\Phi)} Z_\Phi(q,v)$.  Now by \cite{Sokal} one 
has $Z_G(q,v) = \sum_{H \subseteq G} q^{\beta_0(H)} v^{e(H)}$, and therefore
 \begin{align*}
x_G &= \sum_{\Phi \subseteq G} 2^{k-e(\Phi)} (-1)^{k-e(\Phi)} Z_\Phi(q,v) =
(-2)^k \sum_{\Phi \subseteq G} \left(-\frac{1}{2}\right)^{e(\Phi)} 
Z_\Phi(q,v)\\
&= (-2)^k \sum_{\Psi \subseteq \Phi \subseteq G} \left(-\frac{1}{2}\right)^{e(\Phi)} 
q^{\beta_0(\Psi)} v^{e(\Psi)} = (-2)^k \sum_{\Psi \subseteq G} q^{\beta_0(\Psi)} 
v^{e(\Psi)} \sum_{\Phi \supseteq \Psi} \left(-\frac{1}{2}\right)^{e(\Phi)} 
\\ &= (-2)^k \sum_{\Psi \subseteq G} q^{\beta_0(\Psi)} v^{e(\Psi)} 
\left(-\frac{1}{2}\right)^{e(\Psi)} \bigl(1 - \frac{1}{2}\bigr)^{k-e(\Psi)}
= (-1)^k Z_G(q,-v).
 \end{align*}
 \end{proof}

\end{document}